\documentclass{amsart}

\usepackage{amsfonts}
\usepackage{amsmath}
\usepackage{hyperref}

\newtheorem{theorem}{Theorem}
\theoremstyle{plain}

\newtheorem{corollary}{Corollary}

\newtheorem{lemma}{Lemma}

\newtheorem{remark}{Remark}

\numberwithin{equation}{section}

\begin{document}

\title[New Inequalities of Steffensen's type]{New Inequalities of Steffensen's type for $s$--convex functions}

\author[M.W. Alomari]{Mohammad W. Alomari}
\address{Department of Mathematics,
Faculty of Science, Jerash University, 26150 Jerash, Jordan}
\email{mwomath@gmail.com}




\begin{abstract}
In this work, new inequalities connected with the Steffensen's
integral inequality for $s$-convex functions are proved.
\end{abstract}

\maketitle

\section{Introduction}
In order to study certain inequalities between mean values, J.F.
Steffensen \cite{RefK} has proved the following inequality (see
also \cite{RefI}, p. 311):
\begin{theorem}
Let $f$ and $g$ be two integrable functions defined on $(a,b)$,
$f$ is decreasing and for each $t \in (a,b)$, $0 \le g(t) \le 1$.
Then, the following inequality
\begin{eqnarray}
\label{eq1}\int\limits_{b - \lambda }^b {f\left( t \right)dt}  \le
\int\limits_a^b {f\left( t \right)g\left( t \right)dt} \le
\int\limits_a^{a + \lambda } {f\left( t \right)dt}
\end{eqnarray}
holds, where, $\lambda  = \int\limits_a^b {g\left( t \right)dt}$.
\end{theorem}

Some minor generalization of Steffensen's inequality (\ref{eq1})
was considered by T. Hayashi \cite{RefE}, using the substituting $
{{g\left( t \right)} \mathord{\left/
 {\vphantom {{g\left( t \right)} A}} \right.
 \kern-\nulldelimiterspace} A}$ for $g\left( t \right)$, where $A$ is positive constant.
For other result involving Steffensen type inequality, see
\cite{RefC}, \cite{RefE} and \cite{RefH}--\cite{RefK}.\\

 In the recent work \cite{RefA}, Alomari et. al. proved
the following result:
\begin{theorem}
\label{thm5}Let $f,g: [a,b] \to \mathbb{R}$ be integrable such
that $0 \le g(t) \le 1$, for all $t \in [a,b]$ such that $\int_a^b
{g\left( t \right)df\left( t \right)}$ exists. If $f$ is
absolutely continuous on $[a,b]$ with $f' \in L_p[a,b]$, $1 \le p
\le \infty$, then we have
\begin{multline}
\label{eq2.16}\left| {\int_a^{a + \lambda } {f\left( t \right)dt -
\int_a^b {f\left( t \right)g\left( t \right)dt} } } \right|
\\
\le \left\{
\begin{array}{l}
 \frac{1}{2}\left[ {\lambda ^2  + \left( {b - a - \lambda }
\right)^2 } \right] \left\| {f'} \right\|_{\infty ,\left[ {a,b}
\right]},\,\,\,\,\,\,\,\,\,\,\,\,\,\,\,\,\,\,\,\,\,\,\,\,\,\,\,\,\,if\,\,\,\,\,\,\,\,f' \in L_\infty  [a,b]; \\
  \\
 \frac{\left\| {f'} \right\|_{p ,\left[ {a,b}
\right]}}{{\left( {q + 1} \right)^{1/q} }} \left[ {\lambda
^{\left( {q + 1} \right)/q}  + \left( {b - a - \lambda }
\right)^{\left( {q + 1} \right)/q} } \right],\,\,\,\,\,\,\,\,if\,\,\,\,\,\,\,\,f' \in L_p [a,b],p > 1; \\
  \\
 \left[ \int_{a + \lambda }^b {g\left( t \right)dt}
\right]\left\| {f'} \right\|_1 ,\,\,\,\,\,\,\,\,\,\,\,\,\,\,\,\,\,\,\,\,\,\,\,\,\,\,\,\,\,\,\,\,\,\,\,\,\,\,\,\,\,\,\,\,\,\,\,\,\,\,\,\,\,\,\,\,\,\,\,\,\,if\,\,\,\,\,\,\,f' \in L_1 [a,b], \\
 \end{array} \right.
\end{multline}
and
\begin{multline}
\label{eq2.17}\left| {\int_a^b {f\left( t \right)g\left( t
\right)dt - \int_{b - \lambda }^b {f\left( t \right)dt} } }
\right|
\\
\le \left\{
\begin{array}{l}
 \frac{1}{2}\left[ {\lambda ^2  + \left( {b - a - \lambda }
\right)^2 } \right] \left\| {f'} \right\|_{\infty ,\left[ {a,b}
\right]},\,\,\,\,\,\,\,\,\,\,\,\,\,\,\,\,\,\,\,\,\,\,\,\,\,\,\,\,\,if\,\,\,\,\,\,\,\,f' \in L_\infty  [a,b]; \\
  \\
 \frac{\left\| {f'} \right\|_{p ,\left[ {a,b}
\right]}}{{\left( {q + 1} \right)^{1/q} }} \left[ {\lambda
^{\left( {q + 1} \right)/q}  + \left( {b - a - \lambda }
\right)^{\left( {q + 1} \right)/q} } \right],\,\,\,\,\,\,\,\,if\,\,\,\,\,\,\,\,f' \in L_p [a,b],p > 1; \\
  \\
 \left[ \int_a^{b - \lambda } {g\left( t \right)dt}
\right]\left\| {f'} \right\|_1 ,\,\,\,\,\,\,\,\,\,\,\,\,\,\,\,\,\,\,\,\,\,\,\,\,\,\,\,\,\,\,\,\,\,\,\,\,\,\,\,\,\,\,\,\,\,\,\,\,\,\,\,\,\,\,\,\,\,\,\,\,if\,\,\,\,\,\,\,f' \in L_1 [a,b]. \\
 \end{array} \right.
\end{multline}
\end{theorem}

A function $f:\mathbb{R}^{+}\rightarrow \mathbb{R}$, where
$\mathbb{R}^{+}=\left[ {0,\infty }\right) $, is said to be
$s$-convex in the second sense if
\begin{align*}
f\left( {\alpha x+\beta y}\right) \leq \alpha ^{s}f\left( x\right) +{\beta }%
^{s}f\left( y\right)
\end{align*}
for all $x,y\in \left[ {0,\infty }\right) $, $\alpha ,\beta \geq 0$ with $%
\alpha +\beta =1$ and for some fixed $s\in \left( {0,1}\right] $.
This class of $s$-convex functions is usually denoted by
$K_{s}^{2}$, (see \cite{RefF}). It can be easily seen that for
$s=1$, $s$-convexity reduces to the ordinary convexity of
functions defined on $\left[ {0,\infty }\right) $.

In \cite{RefD}, Dragomir and Fitzpatrick proved a variant of
Hadamard's inequality which holds for $s$--convex functions in the
second sense:
\begin{align}
2^{s-1}f\left( {\frac{{a+b}}{2}}\right) \leq \frac{1}{{b-a}}\int_{a}^{b}{%
f\left( x\right) dx}\leq \frac{{f\left( a\right) +f\left( b\right)
}}{{s+1}}. \label{2}
\end{align}
The constant $k=\frac{1}{{s+1}}$ is the best possible in the
second inequality in (\ref{2}). For another inequalities of
Hermite--Hadamard type see \cite{RefB} and \cite{RefG}.

The aim of this paper is to establish new inequalities of
Steffensen's type for $s$-convex functions in the second sense.

\section{The results }

\begin{lemma}
\label{lemma} Let $f,g: [a,b] \subset \mathbb{R}^{+}\to
\mathbb{R}$ be integrable such that $0 \le g(t) \le 1$, for all $t
\in [a,b]$ and $\int_a^b {g\left( t \right)f'\left( t \right)dx}$
exists. Then we have the following representation
\begin{multline}
\label{eq2.1}\int_a^{a + \lambda } {f\left( t \right)dt}  -
\int_a^b {f\left( t \right)g\left( t \right)dt}
\\
= -  \int_a^{a + \lambda } {\left( {\int_a^x {\left( {1 - g\left(
t \right)} \right)dt} } \right)f'\left( x \right)dx} - \int_{a +
\lambda }^b {\left( {\int_x^b {g\left( t \right)dt} }
\right)f'\left( x \right)dx} ,
\end{multline}
and
\begin{multline}
\label{eq2.2} \int_a^b {f\left( t \right)g\left( t \right)dt -
\int_{b - \lambda }^b {f\left( t \right)dt} }
\\
=  - \int_a^{b - \lambda } {\left( {\int_a^x {g(t)dt} }
\right)f'\left( x \right)dx}  - \int_{b - \lambda }^b {\left(
{\int_x^b {\left( {1 - g\left( x \right)} \right)dt} }
\right)f'\left( x \right)dx},
\end{multline}
where $\lambda : = \int_a^b {g\left( t \right)dt}$.
\end{lemma}

\begin{proof}
Integrating by parts
\begin{align*}
&-  \int_a^{a + \lambda } {\left( {\int_a^x {\left( {1 - g\left( t
\right)} \right)dt} } \right)f'\left( x \right)dx} - \int_{a +
\lambda }^b {\left( {\int_x^b {g\left( t \right)dt} }
\right)f'\left( x \right)dx}
\\
&= - \left( {\int_a^{a + \lambda } {\left( {1 - g\left( t \right)}
\right)dt} } \right)f\left( {a + \lambda } \right) + \int_a^{a +
\lambda } {f\left( x \right)d\left( {\int_a^x {\left( {1 - g\left(
t \right)} \right)dt} } \right)}
\\
&\qquad + \left( {\int_{a + \lambda }^b {g\left( t \right)dt} }
\right)f\left( {a + \lambda } \right) + \int_{a + \lambda }^b
{f\left( x \right)d\left( {\int_x^b {g\left( t \right)dt} }
\right)}
\\
&=  - \left( {\int_a^{a + \lambda } {\left( {1 - g\left( t
\right)} \right)dt} } \right)f\left( {a + \lambda } \right) +
\int_a^{a + \lambda } {f\left( x \right)\left( {1 - g\left( x
\right)} \right)dx}
\\
&\qquad + \left( {\int_{a + \lambda }^b {g\left( t \right)dt} }
\right)f\left( {a + \lambda } \right) - \int_{a + \lambda }^b
{f\left( x \right)g\left( x \right)dx}
\\
&=  - \lambda f\left( {a + \lambda } \right) + f\left( {a +
\lambda } \right)\int_a^{a + \lambda } {g\left( t \right)dt}  +
\int_a^{a + \lambda } {f\left( x \right)dx}
\\
&\qquad - \int_a^{a + \lambda } {f\left( x \right)g\left( x
\right)dx} + f\left( {a + \lambda } \right)\int_{a + \lambda }^b
{g\left( t \right)dt}  - \int_{a + \lambda }^b {f\left( x
\right)g\left( x \right)dx}
\\
&=- \lambda f\left( {a + \lambda } \right) + f\left( {a + \lambda
} \right)\int_a^{a + \lambda } {g\left( t \right)dt}  + f\left( {a
+ \lambda } \right)\int_{a + \lambda }^b {g\left( t \right)dt}
\\
&\qquad + \int_a^{a + \lambda } {f\left( x \right)dx}  - \int_a^b
{f\left( x \right)g\left( x \right)dx}
\\
&=\int_a^{a + \lambda } {f\left( x \right)dx}  - \int_a^b {f\left(
x \right)g\left( x \right)dx},
\end{align*}
which gives the desired representation (\ref{eq2.1}). The identity
(\ref{eq2.2}) can be also proved in a similar way, we shall omit
the details.
\end{proof}

\subsection{Inequalities involving $s$-convexity}
In the following, inequalities for absolutely continuous functions
whose first derivatives are $s$-convex ($s$-concave) are given:
\begin{theorem}
\label{thm.main1}Let $f,g: [a,b]\subset \mathbb{R}^{+} \to
\mathbb{R}$ be integrable such that $0 \le g(t) \le 1$, for all $t
\in [a,b]$ such that $\int_a^b {g\left( t \right)f'\left( t
\right)dx}$ exists. If $f$ is absolutely continuous on $[a,b]$
such that $|f'|$ is $s$-convex on $[a,b]$, then
\begin{multline}
\label{eq2.3}\left| {\int_a^{a + \lambda } {f\left( t \right)dt -
\int_a^b {f\left( t \right)g\left( t \right)dt} } } \right|
\\
\le \frac{1}{{\left( {s + 1} \right)\left( {s + 2} \right)}}\left[
{\lambda ^{ 2} \left| {f'\left( a \right)} \right| + \left( {b - a
- \lambda } \right)^{2} \left| {f'\left( b \right)} \right|}
\right]
\\
+ \frac{1}{{s + 2}}\left[ {\lambda ^{ 2}  + \left( {b - a -
\lambda } \right)^{ 2} } \right]\left| {f'\left( {a + \lambda }
\right)} \right|,
\end{multline}
and
\begin{multline}
\label{eq2.4}\left| {\int_a^b {f\left( t \right)g\left( t
\right)dt}  - \int_{b - \lambda }^b {f\left( t \right)dt} }
\right|
\\
\le \frac{1}{{\left( {s + 1} \right)\left( {s + 2} \right)}}\left[
{\lambda ^{ 2} \left| {f'\left( b \right)} \right| + \left( {b - a
- \lambda } \right)^{2} \left| {f'\left( a \right)} \right|}
\right]
\\
+ \frac{1}{{s + 2}}\left[ {\lambda ^{ 2}  + \left( {b - a -
\lambda } \right)^{2} } \right]\left| {f'\left( {b - \lambda }
\right)} \right|,
\end{multline}
where, $\lambda : = \int_a^b {g\left( t \right)dt}$.
\end{theorem}

\begin{proof}
Utilizing the\ triangle inequality on (\ref{eq2.1}), and since
$|f'|$ is $s$-convex, we have
\begin{align*}
&\left| {\int_a^{a + \lambda } {f\left( t \right)dt}  - \int_a^b
{f\left( t \right)g\left( t \right)dt} } \right|
\\
&\le \left| {\int_a^{a + \lambda } {\left( {\int_a^x {\left( {1 -
g\left( t \right)} \right)dt} } \right)f'\left( x \right)dx} }
\right|+ \left| {\int_{a + \lambda }^b {\left( {\int_x^b {g\left(
t \right)dt} } \right)f'\left( x \right)dx} } \right|
\\
&\le \int_a^{a + \lambda } { \left| {\int_a^x {\left( {1 - g\left(
t \right)} \right)dt} } \right|\left| {f'\left( x \right)}
\right|dx} + \int_{a + \lambda }^b {\left| {\int_x^b {g\left( t
\right)dt} } \right| \left| {f'\left( x \right)} \right| dx}
\end{align*}
\begin{align*}
&\le \int_a^{a + \lambda } { \left| {\int_a^x {\left( {1 - g\left(
t \right)} \right)dt} } \right|\left[ {\frac{\left( {x -
a}\right)^s}{\lambda^s} \left| {f'\left( {a+\lambda} \right)}
\right| + \frac{\left( {a+ \lambda - x}\right)^s}{\lambda^s}
\left| {f'\left( a \right)} \right|} \right]dx}
\\
&\qquad+ \int_{a + \lambda }^b {\left| {\int_x^b {g\left( t
\right)dt} } \right| \left[ {\frac{\left( {x - a -
\lambda}\right)^s}{\left( {b - a - \lambda} \right)^s} \left|
{f'\left( b \right)} \right| + \frac{\left( {b -
x}\right)^s}{\left( {b - a - \lambda} \right)^s} \left| {f'\left(
a +\lambda \right)} \right|} \right] dx}
\\
&\le  \frac{{\left| {f'\left( a+\lambda \right)}
\right|}}{{\lambda^s}} \int_a^{a + \lambda } { \left( {\int_a^x
{\left| {1 - g\left( t \right)} \right|dt}} \right) \left( {x - a}
\right)^s dx}
\\
&\qquad+\frac{{\left| {f'\left( a \right)} \right|}}{{\lambda^s}}
\int_a^{a + \lambda } { \left( {\int_a^x {\left| {1 - g\left( t
\right)} \right|dt}} \right) \left( {a + \lambda - x} \right)^s
dx}
\\
&\qquad+ \frac{{\left| {f'\left( b \right)} \right|}}{{\left( {b -
a - \lambda} \right)^s}} \int_{a + \lambda }^b {\left( {\int_x^b
{\left| {g\left( t \right)} \right|dt} } \right) \left( {x - a
-\lambda} \right)^s dx}
\\
&\qquad+ \frac{{\left| {f'\left( a + \lambda \right)}
\right|}}{{\left( {b - a - \lambda} \right)^s}} \int_{a + \lambda
}^b {\left( {\int_x^b {\left| {g\left( t \right)} \right|dt} }
\right)\left( {b - x} \right)^sdx}
\\
&\le \frac{{\left| {f'\left( a+\lambda \right)}
\right|}}{{\lambda^s}} \int_a^{a + \lambda } { \left( {x - a}
\right)^{s+1} dx} +\frac{{\left| {f'\left( a \right)}
\right|}}{{\lambda^s}} \int_a^{a + \lambda } { \left( {x - a}
\right) \left( {a + \lambda - x} \right)^s dx}
\\
&\qquad+ \frac{{\left| {f'\left( b \right)} \right|}}{{\left( {b -
a - \lambda} \right)^s}} \int_{a + \lambda }^b {\left( {b - x}
\right)\left( {x - a -\lambda} \right)^sdx} + \frac{{\left|
{f'\left( a + \lambda \right)} \right|}}{{\left( {b - a - \lambda}
\right)^s}} \int_{a + \lambda }^b {\left( {b - x} \right)^{s+1}
dx}
\\
&=\frac{1}{{\left( {s + 1} \right)\left( {s + 2} \right)}}\left[
{\lambda ^{2} \left| {f'\left( a \right)} \right| + \left( {b - a
- \lambda } \right)^{2} \left| {f'\left( b \right)} \right|}
\right]
\\
&\qquad+ \frac{1}{{s + 2}}\left[ {\lambda ^{ 2}  + \left( {b - a -
\lambda } \right)^{ 2} } \right]\left| {f'\left( {a + \lambda }
\right)} \right|
\end{align*}
which proves the first inequality in (\ref{eq2.3}). In similar way
and using (\ref{eq2.2}) we may deduce the desired inequality
(\ref{eq2.4}), and we shall omit the details.
\end{proof}

\begin{corollary} \label{cor1}
In (\ref{eq2.3}) if one chooses $s=1$ then
\begin{multline}
\label{eq2.5}\left| {\int_a^{a + \lambda } {f\left( t \right)dt -
\int_a^b {f\left( t \right)g\left( t \right)dt} } } \right|
\\
\le  \frac{1}{6}\lambda ^2\left| {f'\left( a \right)} \right|
+\frac{1}{3} \left[ {\lambda ^2+ \left( {b - a - \lambda }
\right)^2}\right]\left| {f'\left( a + \lambda \right)} \right| +
\frac{1}{6}\left( {b - a - \lambda } \right)^2 \left| {f'\left( b
\right)} \right|
\end{multline}
also, in (\ref{eq2.4}) if $s=1$, then
\begin{multline}
\label{eq2.6}\left| {\int_a^b {f\left( t \right)g\left( t
\right)dt}  - \int_{b - \lambda }^b {f\left( t \right)dt} }
\right|
\\
\le  \frac{1}{6}\lambda ^2\left| {f'\left( b \right)} \right|
+\frac{1}{3} \left[ {\lambda ^2+ \left( {b - a - \lambda }
\right)^2}\right]\left| {f'\left( b - \lambda \right)} \right| +
\frac{1}{6}\left( {b - a - \lambda } \right)^2 \left| {f'\left( a
\right)} \right|
\end{multline}
\end{corollary}

\begin{remark}
In the inequalities (\ref{eq2.3}) and (\ref{eq2.4}), choose
$\lambda = 0$, then we have
\begin{multline}
\label{eq2.7}\left| {\int_a^b {f\left( t \right)g\left( t
\right)dt}  } \right|
\\
\le  \frac{\left( {b - a} \right)^{2}}{{\left( {s + 1}
\right)\left( {s + 2} \right)}}\min\left\{ {\left( {s + 1}
\right)\left| {f'\left( a \right)} \right|  + \left| {f'\left( b
\right)} \right| ,\left| {f'\left( a \right)} \right| + \left( {s
+ 1} \right)\left| {f'\left( b \right)} \right| } \right\}.
\end{multline}
\end{remark}

Another approach leads to the following result:
\begin{theorem}
\label{thm.main2}Let $f,g: [a,b]\subset \mathbb{R}^{+} \to
\mathbb{R}$ be integrable such that $0 \le g(t) \le 1$, for all $t
\in [a,b]$ such that $\int_a^b {g\left( t \right)f'\left( t
\right)dx}$ exists. If $f$ is absolutely continuous on $[a,b]$
with $|f'|$ is $s$-convex on $[a,b]$, for some fixed $s\in (0,1]$
then we have
\begin{multline}
\left| {\int_a^{a + \lambda } {f\left( t \right)dt}  - \int_a^b
{f\left( t \right)g\left( t \right)dt} } \right|
\\
\le \frac{1}{{s + 1}}\left[ {\int_{a + \lambda }^b {g\left( t
\right)dt}} \right] \cdot\left[ {\lambda \left| {f'\left( a
\right)} \right| + \left( {b - a} \right)\left| {f'\left( {a +
\lambda } \right)} \right| + \left( {b - a - \lambda }
\right)\left| {f'\left( b \right)} \right|} \right]\label{eq2.8}
\\
\le \frac{\left( {b - a - \lambda } \right)}{{s + 1}}\left[
{\lambda \left| {f'\left( a \right)} \right| + \left( {b - a}
\right)\left| {f'\left( {a + \lambda } \right)} \right| + \left(
{b - a - \lambda } \right)\left| {f'\left( b \right)} \right|}
\right],
\end{multline}
and
\begin{multline}
\left| {\int_a^b {f\left( t \right)g\left( t \right)dt}  - \int_{b
- \lambda }^b {f\left( t \right)dt} } \right|
\label{eq2.9}\\
\le \frac{1}{{s + 1}}\left[ {\int_{b - \lambda}^b {g\left( t
\right)dt}} \right] \cdot\left[ {\left( {b - a - \lambda }
\right)\left| {f'\left( a \right)} \right| + \left( {b - a}
\right)\left| {f'\left( {b - \lambda } \right)} \right| + \lambda
\left| {f'\left( b \right)} \right|} \right]
\\
\le \frac{\lambda}{{s + 1}} \left[ {\left( {b - a - \lambda }
\right)\left| {f'\left( a \right)} \right| + \left( {b - a}
\right)\left| {f'\left( {b - \lambda } \right)} \right| + \lambda
\left| {f'\left( b \right)} \right|} \right],
\end{multline}
where $\lambda : = \int_a^b {g\left( t \right)dt}$.
\end{theorem}

\begin{proof}
From Lemma \ref{lemma}, we may write
\begin{align*}
&\left| {\int_a^{a + \lambda } {f\left( t \right)dt}  - \int_a^b
{f\left( t \right)g\left( t \right)dt} } \right|
\\
&\le \mathop {\sup }\limits_{x \in \left[ {a,a + \lambda }
\right]} \left[ {\int_a^x {\left( {1 - g\left( t \right)}
\right)dt} } \right] \cdot \int_a^{a + \lambda } \left|
f'\left(x\right) \right|dx + \mathop {\sup }\limits_{x \in \left[
{a + \lambda ,b} \right]} \left[ {\int_x^b {g\left( t \right)dt} }
\right] \cdot \int_{a + \lambda }^b
\left|f'\left(x\right)\right|dx.
\end{align*}
Since $|f'|$ is $s$--convex on $[a,b]$, then by (\ref{2}) we have
\begin{align*}
\int_a^{a + \lambda}  \left|f'\left(x\right)\right|dx \le
\lambda\cdot \frac{{\left| {f'\left( a \right)} \right| + \left|
{f'\left( {a + \lambda } \right)} \right|}}{{s + 1}},
\end{align*}
and
\begin{align*}
\int_{a + \lambda }^b \left|f'\left(x\right)\right|dx \le \left(
{b - a - \lambda } \right)\cdot\frac{{\left| {f'\left( {a +
\lambda } \right)} \right| + \left| {f'\left( b \right)}
\right|}}{{s + 1}}.
\end{align*}
Therefore, we have
\begin{align*}
&\left| {\int_a^{a + \lambda } {f\left( t \right)dt}  - \int_a^b
{f\left( t \right)g\left( t \right)dt} } \right|
\\
&\le\lambda\cdot \frac{{\left| {f'\left( a \right)} \right| +
\left| {f'\left( {a + \lambda } \right)} \right|}}{{s + 1}} \cdot
\int_a^{a + \lambda } {\left( {1 - g\left( t \right)} \right)dt}
\\
&\qquad+ \left( {b - a - \lambda } \right)\cdot\frac{{\left|
{f'\left( {a + \lambda } \right)} \right| + \left| {f'\left( b
\right)} \right|}}{{s + 1}}
 \cdot \int_{a + \lambda }^b {g\left( t \right)dt}
\\
&\le \max \left\{ {\int_a^{a + \lambda } {\left( {1 - g\left( t
\right)} \right)dt} ,\int_{a + \lambda }^b {g\left( t \right)dt} }
\right\} \cdot \left[ {\lambda \frac{{\left| {f'\left( a \right)}
\right| + \left| {f'\left( {a + \lambda } \right)} \right|}}{{s +
1}}}\right.
\\
&\qquad\left. {+ \left( {b - a - \lambda } \right)\frac{{\left|
{f'\left( {a + \lambda } \right)} \right| + \left| {f'\left( b
\right)} \right|}}{{s + 1}}}\right]
\\
&= \frac{1}{{s + 1}}\left[ {\int_{a + \lambda }^b {g\left( t
\right)dt}} \right] \cdot\left[ {\lambda \left| {f'\left( a
\right)} \right| + \left( {b - a} \right)\left| {f'\left( {a +
\lambda } \right)} \right| + \left( {b - a - \lambda }
\right)\left| {f'\left( b \right)} \right|} \right],
\end{align*}
which proves the first inequality in (\ref{eq2.8}). The second
inequality in (\ref{eq2.8}) follows directly, since $0 \le g(t)
\le 1$ for all $t \in [a,b]$, then
\begin{align*}
0 \le \int_{a + \lambda }^b {g\left( t \right)dt} &\le \left( {b -
a - \lambda } \right).
\end{align*}
The inequalities in (\ref{eq2.9}) may be proved in the same way
using the identity (\ref{eq2.2}), we shall omit the details.
\end{proof}

\subsection{Inequalities involving $s$-concavity}
\begin{theorem}
\label{thm5}Let $f,g: [a,b]\subset \mathbb{R}^{+} \to \mathbb{R}$
be integrable such that $0 \le g(t) \le 1$, for all $t \in [a,b]$
such that $\int_a^b {g\left( t \right)f'\left( t \right)dx}$
exists. If $f$ is absolutely continuous on $[a,b]$ with $|f'|$ is
$s$-concave on $[a,b]$, for some fixed $s\in (0,1]$ then we have
\begin{align}
&\left| {\int_a^{a + \lambda } {f\left( t \right)dt}  - \int_a^b
{f\left( t \right)g\left( t \right)dt} } \right|
\nonumber\\
&\le 2^{s-1}\left[ \int_{a + \lambda }^b {g\left( t \right)dt}
\right]\cdot \left[ {\lambda\left| {f'\left( {a + \frac{\lambda
}{2}} \right)} \right| + \left( {b - a - \lambda} \right)\left|
{f'\left( {\frac{{a + b + \lambda }}{2}} \right)} \right|}
\right]. \label{eq2.10}
\\
&\le 2^{s-1}\left( {b - a - \lambda} \right)\cdot \left[
{\lambda\left| {f'\left( {a + \frac{\lambda }{2}} \right)} \right|
+ \left( {b - a - \lambda} \right)\left| {f'\left( {\frac{{a + b +
\lambda }}{2}} \right)} \right|} \right] \nonumber
\end{align}
and
\begin{align}
&\left| {\int_a^b {f\left( t \right)g\left( t \right)dt}-\int_{b -
\lambda }^b {f\left( t \right)dt}} \right|
\nonumber\\
&\le 2^{s-1}\left[ \int_{b - \lambda }^b {g\left( t \right)dt}
\right]\cdot \left[ { \left( {b - a - \lambda} \right)\left|
{f'\left( {\frac{{a + b - \lambda }}{2}} \right)} \right| +
\lambda\left| {f'\left( {b - \frac{\lambda }{2}} \right)} \right|
} \right]. \label{eq2.11}
\\
&\le \lambda 2^{s-1}\cdot \left[ { \left( {b - a - \lambda}
\right)\left| {f'\left( {\frac{{a + b - \lambda }}{2}} \right)}
\right| + \lambda\left| {f'\left( {b - \frac{\lambda }{2}}
\right)} \right| } \right], \nonumber
\end{align}
where $\lambda : = \int_a^b {g\left( t \right)dt}$.
\end{theorem}

\begin{proof}
Utilizing the triangle inequality on (\ref{eq2.1}), and since
$|f'|$ is $s$--concave on $[a,b]$ then by (\ref{2}) we may state
\begin{align*}
&\left| {\int_a^{a + \lambda } {f\left( t \right)dt}  - \int_a^b
{f\left( t \right)g\left( t \right)dt} } \right|
\\
&\le \mathop {\sup }\limits_{x \in \left[ {a,a + \lambda }
\right]} \left[ {\int_a^x {\left( {1 - g\left( t \right)}
\right)dt} } \right] \cdot \int_a^{a + \lambda } \left|
f'\left(x\right) \right|dx + \mathop {\sup }\limits_{x \in \left[
{a + \lambda ,b} \right]} \left[ {\int_x^b {g\left( t \right)dt} }
\right] \cdot \int_{a + \lambda }^b
\left|f'\left(x\right)\right|dx
\\
&= 2^{s-1}\lambda\left| {f'\left( {a + \frac{\lambda }{2}}
\right)} \right|\cdot \int_a^{a + \lambda } {\left( {1 - g\left( t
\right)} \right)dt} + 2^{s-1} \left( {b - a - \lambda } \right)
\left| {f'\left( {\frac{{a + b + \lambda }}{2}} \right)} \right|
 \cdot \int_{a + \lambda }^b {g\left( t \right)dt}
\\
&= 2^{s-1}\left[ \int_{a + \lambda }^b {g\left( t \right)dt}
\right]\cdot \left[ {\lambda\left| {f'\left( {a + \frac{\lambda
}{2}} \right)} \right| + \left( {b - a - \lambda} \right)\left|
{f'\left( {\frac{{a + b + \lambda }}{2}} \right)} \right|}
\right].
\end{align*}
which proves the first inequality in (\ref{eq2.10}). The second
inequality in (\ref{eq2.10}) follows directly, since $0 \le g(t)
\le 1$ for all $t \in [a,b]$, then
\begin{align*}
0 \le \int_{a + \lambda }^b {g\left( t \right)dt} &\le \left( {b -
a - \lambda } \right).
\end{align*}
The inequalities in (\ref{eq2.11}) may be proved in the same way
using the identity (\ref{eq2.2}), we shall omit the details.
\end{proof}

Another result is incorporated in the following theorem:
\begin{theorem}
\label{thm.main3}Let $f,g: [a,b]\subset \mathbb{R}^{+} \to
\mathbb{R}$ be integrable such that $0 \le g(t) \le 1$, for all $t
\in [a,b]$ such that $\int_a^b {g\left( t \right)f'\left( t
\right)dx}$ exists. If $f$ is absolutely continuous on $[a,b]$
with $|f'|^q$ is $s$-concave on $[a,b]$, for some fixed $s\in
(0,1]$ and $q>1$, then we have
\begin{multline}
\label{eq2.12}\left| {\int_a^{a + \lambda } {f\left( t \right)dt -
\int_a^b {f\left( t \right)g\left( t \right)dt} } } \right|
\\
\le 2^{\left( {s - 1} \right)/q} \left[ { \lambda^2 \left|
{f'\left( {a + \frac{\lambda }{2}} \right)} \right|  + \left( {b -
a - \lambda } \right)^2 \left| {f'\left( {\frac{{a + b + \lambda
}}{2}} \right)} \right| } \right],
\end{multline}
and
\begin{multline}
\label{eq2.13}\left| {\int_a^b {f\left( t \right)g\left( t
\right)dt} - \int_{b - \lambda }^b {f\left( t \right)dt } }
\right|
\\
\le 2^{\left( {s - 1} \right)/q} \left[ { \left( {b - a - \lambda
} \right)^2 \left| {f'\left( {b - \frac{\lambda }{2}} \right)}
\right|  +  \lambda^2\left| {f'\left( {\frac{{a + b - \lambda
}}{2}} \right)} \right|  } \right],
\end{multline}
where $\lambda : = \int_a^b {g\left( t \right)dt}$.
\end{theorem}

\begin{proof}
From Lemma \ref{lemma} and using the H\"{o}lder inequality for $q
> 1$, and $p = \frac{q}{q-1}$, we obtain
\begin{align}
&\left| {\int_a^{a + \lambda } {f\left( t \right)dt}  - \int_a^b
{f\left( t \right)g\left( t \right)dt} } \right|
\nonumber\\
&\le \int_a^{a + \lambda } { \left| {\int_a^x {\left( {1 - g\left(
t \right)} \right)dt} } \right|\left| {f'\left( x \right)}
\right|dx} + \int_{a + \lambda }^b {\left| {\int_x^b {g\left( t
\right)dt} } \right| \left| {f'\left( x \right)} \right| dx}
\nonumber\\
&\le  \left( {\int_a^{a + \lambda } { \left| {\int_a^x {\left( {1
- g\left( t \right)} \right)dt} } \right|^p dx}} \right)^{1/p}
\left( {\int_a^{a + \lambda } {\left| {f'\left( x \right)}
\right|^q dx}} \right)^{1/q}
\label{eq2.14}\\
&\qquad+ \left( {\int_{a + \lambda }^b {\left| {\int_x^b {g\left(
t \right)dt} } \right|^p dx}} \right)^{1/p} \left( {\int_{a +
\lambda }^b { \left| {f'\left( x \right)} \right|^q dx}}
\right)^{1/q}:=M,\nonumber
\end{align}
where $p$ is the conjugate of $q$.

By the inequality (\ref{2}), we have
\begin{align*}
\int_a^{a + \lambda } {\left| {f'\left( x \right)} \right|^q dx}
\le 2^{s - 1} \lambda \left| {f'\left( {a + \frac{1}{2}\lambda }
\right)} \right|^q,
\end{align*}
and
\begin{align*}
\int_{a + \lambda }^b {\left| {f'\left( x \right)} \right|^q dx}
\le 2^{s - 1} \left( {b - a - \lambda} \right) \left| {f'\left(
{\frac{{a + b + \lambda }}{2}} \right)} \right|^q,
\end{align*}
which gives by (\ref{eq2.14})
\begin{align*}
M &\le 2^{\left( {s - 1} \right)/q} \lambda^{1/q}\left| {f'\left(
{a + \frac{\lambda }{2}} \right)} \right| \left( {\int_a^{a +
\lambda } { \left( {x - a} \right)^pdx} } \right)^{1/p}
\\
&\qquad+ 2^{\left( {s - 1} \right)/q} \left( {b - a - \lambda}
\right)^{1/q}\left| {f'\left( {\frac{{a + b + \lambda }}{2}}
\right)} \right| \left( {\int_{a + \lambda }^b {\left( {b - x }
\right)^pdx}} \right)^{1/p}
\\
&=  2^{\left( {s - 1} \right)/q} \left[ { \lambda ^{1 +
{\textstyle{1 \over p}} + {\textstyle{1 \over q}}} \left|
{f'\left( {a + \frac{\lambda }{2}} \right)} \right|  + \left( {b -
a - \lambda } \right)^{1 + {\textstyle{1 \over p}} + {\textstyle{1
\over q}}} \left| {f'\left( {\frac{{a + b + \lambda }}{2}}
\right)} \right| } \right]
\\
&=2^{\left( {s - 1} \right)/q} \left[ { \lambda^2 \left| {f'\left(
{a + \frac{\lambda }{2}} \right)} \right|  + \left( {b - a -
\lambda } \right)^2 \left| {f'\left( {\frac{{a + b + \lambda
}}{2}} \right)} \right| } \right],
\end{align*}
giving the second inequality in (\ref{eq2.12}).

The inequalities in (\ref{eq2.13}) may be proved in the same way
using the identity (\ref{eq2.2}), we shall omit the details.
\end{proof}

\begin{remark}
The interested reader may obtain several inequalities for
log-convex, quasi-convex, $r$-convex and $h$-convex functions by
replacing the condition on $|f'|$.
\end{remark}

\end{document}